\documentclass [11pt,a4paper]{article}
\usepackage[T1]{fontenc}
\usepackage[ansinew]{inputenc}
\usepackage{amssymb}
\usepackage{amsmath}
\usepackage{graphicx}
\usepackage{amsthm}
\usepackage{cite}

\def\la{\langle}
\def\ra{\rangle}
\def\cL{\mathcal{L}}
\def\cH{\mathcal{H}}

 \newtheorem{theorem}{Theorem}[section]
 \newtheorem{corollary}{Corollary}[section]
 \newtheorem{lemma}{Lemma}[section]
 \newtheorem{proposition}{Proposition}[section]
 \newtheorem{definition}{Definition}[section]
 \newtheorem{remark}{Remark}[section]
 \numberwithin{equation}{section}

\newcommand{\beq}{\begin{equation}}
\newcommand{\eeq}{\end{equation}}

\begin{document}
%----------------------------------------------------------------------------------------------------------title
\title{  Finite dimensional global and exponential attractors
for a coupled time-dependent Ginzburg-Landau equations for atomic
Fermi gases near the BCS-BEC crossover}
\author{Jie {\sc Jiang}\thanks{Institute of Applied Physics and
Computational Mathematics, PO Box 8009, Beijing 100088, \ P. R.
China, \textsl{jiangbryan@gmail.com}.} , Hao {\sc
Wu}\thanks{Shanghai Key Laboratory for Contemporary Applied
Mathematics and School of Mathematical Sciences,\ Fudan University,
Han Dan Road No. 220, Shanghai 200433,\ P. R. China,
\textsl{haowufd@yahoo.com}. Corresponding author.}\ \ \ and \ Boling
{\sc Guo}
\thanks{Institute of Applied Physics and Computational Mathematics, PO Box 8009, Beijing 100088, \ P. R. China,
\textsl{gbl@iapcm.com}.}}

\date{\today}
\maketitle
\begin{abstract} We study a coupled nonlinear evolution system arising from the Ginzburg-Landau
theory for atomic Fermi gases near the BCS-BEC crossover.  First, we
prove that the initial boundary value problem generates a strongly
continuous semigroup on a suitable phase-space which possesses the
global attractor. Then we establish the existence of an exponential
attractor. As a consequence, we show that the global attractor is of
finite fractal dimension.
\end{abstract}

\section{Introduction }

The superfluidity in the ultra-cold atomic Fermi gases has been paid
much attention by many researchers in recent years, since it
provides a useful testbed for the study of high-temperature
superconductivity in strongly correlated fermionic systems. In the
superfluid atomic Fermi gases near the Feshbach resonance the strong
attractive interaction is realized between fermion atoms which can
cause a crossover from the weak-coupling BCS state to the
strong-coupling BEC one \cite{OG,MK}. The Ginzburg-Landau theory
plays an important role in superconductivity research which was
applied in the pioneering works \cite{DZ,MRE} and later in the
single-component fermion system (single channel model)
\cite{BP,TWD}. Recently, Machida \& Koyama \cite{MK} developed a
time-dependent Ginzburg-Landau (TDGL) theory for the superfluid
atomic Fermi gases near the Feshbach resonance from the
fermion-boson model on the basis of the functional integral
formalism. This two-component TDGL model describes the dynamics of
the superfluid atomic Fermi gases, in which BCS pairs and tightly
bound diatomic condensate coexist. The resulting system consists of
a nonlinear time-dependent complex Ginzburg-Landau equation coupled
with a Schr\"{o}dinger type equation, which reads as follows (cf.
\cite{MK})
 \beq
 \label{1}\begin{cases}-id u_t=\left(-\frac{dg^2+1}{U}+a\right)u+g[a+d(2\nu-2\mu)]
 \phi+\frac{c}{4m}\Delta u+\frac{g}{4m}(c-d)\Delta \phi\\ \quad \quad \quad \quad -b|u+g\phi|^2(u+g\phi),
\\
i\phi_t=-\frac{g}{U}u+(2\nu-2\mu)\phi-\frac{1}{4m}\Delta\phi.\end{cases}
 \eeq
$u$ and $\phi$ are both complex-valued unknown functions, which
stand for the fermion-pair field and the condensed boson field,
respectively. $2\nu$ is the threshold energy of the Feshbach
resonance while $g$ is the coupling constant in the Feshbach
resonance. $\mu$ is the chemical potential and $U>0$ denotes the BCS
coupling constant. The coefficients $a,b$ and $c$ correspond to the
Ginzburg-Landau coefficients in the TDGL theory. All these seven
coefficients are real numbers. The coefficient $d$ is generally
complex, which dominates the dynamics of the superfluid atomic Fermi
gases. In the BCS limit, $d$ can be considered to be purely
imaginary while in the BEC region, the imaginary part of $d$ usually
vanishes. In the BCS-BEC crossover region, both the real and
imaginary parts of $d$ have finite values that $d$ is a complex
number (see e.g., (A3) below). For the detailed discussions on these
physical coefficients, we refer to \cite{MK}.

By introducing a new variable
 $$v=u+g\phi,$$
  we can transform the original system \eqref{1}
  into the following form,
  which is more convenient to be treated from the mathematical point of view (cf. \cite{Chen1,Chen2,FJG09}):
 \beq
 \label{2}
 \begin{cases}d v_t-\left(a-\frac{1}{U}\right)iv-\frac{i g}{U}\phi-\frac{ic}{4m}\Delta v+ib|v|^2v=0,\\
 \phi_t-\frac{i g}{U}v+\frac{ig^2}{U}\phi+i(2\nu-2\mu)\phi-\frac{i}{4m}\Delta\phi=0.
 \end{cases}
 \eeq
  In Chen \& Guo \cite{Chen2}, the authors proved the existence and uniqueness
  of weak solutions to \eqref{2} subject to periodic boundary
  conditions.  Later in \cite{Chen3}, the global
  existence of weak solutions to the periodic boundary value problem of system \eqref{2} with
  a general nonlinearity $ib|v|^pv$ was obtained for certain power $p$ instead of $2$. As far as the
  classical solution is concerned, Chen \& Guo \cite{Chen1}
   studied the initial boundary problem of \eqref{2} subject
   to homogeneous Dirichlet/Neumann boundary conditions for arbitrary spatial dimension.
   They proved the global existence and uniqueness of classical solutions under some specific
   restrictions on the complex coefficient $d$. However, no results on the long-time behavior of the
   global weak/classical solutions to \eqref{2} were obtained in the papers \cite{Chen2,Chen1,Chen3} mentioned
   above. One possible difficulty is that we do not have enough
   dissipative mechanism in the equation for $\phi$. As a first step for the study of the long-time dynamics of the
problem, in the present paper,
 we consider the system \eqref{2} with
a linear weak dissipation in the equation for $\phi$, that is
 \beq\label{3}
 \begin{cases}dv_t-(a-\frac{1}{U})iv-\frac{i g}{U}\phi-\frac{ic}{4m}\Delta v+ib|v|^2v=f,\quad (x,t)\in \Omega \times \mathbb{R}_+,\\
 \phi_t+\gamma \phi-\frac{ig}{U}v+\frac{ig^2}{U}\phi+i(2\nu-2\mu)\phi-\frac{i}{4m}\Delta\phi=h,\quad (x,t)\in \Omega \times \mathbb{R}_+,
 \end{cases}
 \eeq
 where $\gamma>0$ is the damping parameter. For the sake of simplicity, we consider the problem
 in a bounded domain $\Omega \subset\mathbb{R}^3$ whose boundary $\Gamma$ is smooth.
 $f$ and $h$ are given external forces. System \eqref{3} is subject to the homogeneous Dirichlet boundary conditions
 \beq\label{6}
  v=\phi=0,\qquad (x,t)\in \Gamma \times \mathbb{R}_+,
 \eeq
 and the initial conditions
 \beq\label{4} v|_{t=0}=v_0(x),\quad \phi|_{t=0}=\phi_0(x),\qquad x\in \Omega.
 \eeq

To formulate our results, we first introduce some notions on the
functional settings. Let $\cL^2(\Omega)$ (or $L^2(\Omega)$) be the
Lebesgue space of complex-valued (real-valued) functions. By
$(\cdot,\cdot)$ and $\|\cdot\|$, we denote the scalar product and
the norm in $\cL^2(\Omega)$ (or $L^2(\Omega)$), respectively:
$$ (w_1,w_2)=\int_{\Omega} w_1\overline{w_2}dx,\quad \| w\|=\sqrt{(w,w)}.$$
Let $W^{k,p}(\Omega)$ be the standard Sobolev spaces for real-valued
functions and as usual, $H^k(\Omega)=W^{k,2}(\Omega)$.
Correspondingly, Sobolev spaces of complex-valued functions are
denoted by $\mathcal{W}^k(\Omega)$ and similarly,
$\cH^k(\Omega)=\mathcal{W}^{k,2}(\Omega)$. We note that
$L^p(\Omega)=W^{0,p}(\Omega)$,
$\cL^p(\Omega)=\mathcal{W}^{0,p}(\Omega)$ and $\cH^{-1}(\Omega)$ (or
$H^{-1}(\Omega)$) is the dual space of $\cH^1_0(\Omega)$ (or
$H^1_0(\Omega)$).

Let $A$ be the unbounded linear operator defined by $ A=-\Delta$,
whose domain is $D(A)=H^2(\Omega)\cap H^1_0(\Omega)$. It is
well-known that (cf. e.g., \cite{Temam}) one can define spaces
$D(A^s)$ for $s\in\mathbb{R}$, with inner product $\la \cdot,
\cdot\ra_{s}=(A^\frac {s}{2}\cdot, A^\frac {s}{2}\cdot)$ and
corresponding norm $|\cdot|_s=\sqrt{\la \cdot, \cdot\ra_{s}}$. In
particular, $D({A^\frac12})=H_0^1(\Omega),\ D({A^0})=L^2(\Omega),\
D({A^{-\frac12}})=H^{-1}(\Omega)$. We note that corresponding
results hold for the complex-valued functional spaces.

In this paper,  we make the following assumptions on external forces
$f$, $h$ and the coefficients of system \eqref{3}:

(A1) $f,h\in \cH^1_0(\Omega)$ are independent of time,

(A2) $U>0, b>0, c>0, m>0$, $aU<1$, $\gamma>0$,

(A3) $d:=d_r+i d_i$ where $d_r, d_i\in \mathbb{R}$ and $d_i>0$.
$|d|=\sqrt{d_r^2+d_i^2}$.\\

Next, we introduce the weak formulation of problem
\eqref{3}-\eqref{4}:
 \begin{definition}
 A pair of complex-valued functions $(v,\phi)$ is called a weak solution
 to problem \eqref{3}-\eqref{4} in $Q_T:= \Omega \times [0,T]$
 for arbitrary $T>0$, if
 \beq v,\phi\in C([0,T], \cH^1_0(\Omega)), \quad v_t\in L^2((0,T),
 \cL^2(\Omega)),\quad \phi_t\in L^2((0,T),\cH^{-1}(\Omega)),\nonumber
 \eeq
and for arbitrary complex-valued functions $\psi \in
\cH^1_0(\Omega)$ and $\xi\in C^1[0,T]$ with $\xi(T)=0$, it holds
 \begin{eqnarray}
 && \int_0^T \left[-d(v,\xi_t\psi) -i\left(a-\frac{1}{U}\right)( v, \xi \psi) -\frac{i g}{U}(\phi,\xi \psi)
 +\frac{ic}{4m} (\nabla v,\xi \nabla \psi)+ib(|v|^2v,\xi
 \psi)\right]dt\nonumber\\
 &=& \int_0^T( f, \xi w)dt+ (v_0, \xi(0)\psi), \label{w1}
 \\
 && \int_0^T\left[ -(\phi,\xi_t\psi)-\frac{i g}{U}(v,\xi\psi)+i\left(\frac{g^2}{U}+2\nu-2\mu\right)
 (\phi,\xi\psi)+\frac{i}{4m}(\nabla\phi,\xi\nabla\psi)+\gamma(\phi,\xi\psi)\right]\nonumber\\
 &=&\int_0^T ( h,\xi\psi)dt+ (\phi_0, \xi(0)\psi). \label{w2}
 \end{eqnarray}
\end{definition}
\noindent The main results of this paper are as follows:

 (a)
Existence and uniqueness of global weak solutions (cf. Theorem
\ref{T2.1} and Corollary \ref{uniw});

 (b) Existence of a global attractor with
finite fractal dimension (cf. Theorem \ref{GA} and Corollary
\ref{GA1});

 (c) Existence of an exponential attractor (cf. Theorem
\ref{EPA}).

We note that in the recent paper \cite{FJG09}, the authors also
considered the long-time behavior of system \eqref{1} with a linear
weak dissipation term in the equation for $\phi$. In particular,
they proved the existence of a weakly compact attractor under some
specific restrictions on the coefficients $\gamma$, $g$, $c$ and $d$
when the spatial dimension is three. However, comparing their
results, our present work has some new features. (i) We prove the
existence of an absorbing set in $\cH_0^1\times \cH_0^1$ for our
problem \eqref{3}-\eqref{4} under much simpler assumptions on the
physical coefficients (cf. (A2), (A3)). In \cite{FJG09}, the
corresponding result was obtained under some rather specific
restrictions on the coefficients. For instance, it was required that
$0<g<2$ and the positive damping parameter $\gamma$, denoted by
$\beta$ in \cite{FJG09}, was assumed to be bounded from below by a
positive constant such that
$\beta>\frac{\frac{g}{|d|U}+\frac{g}{U}+1}{2-g}>0$.  Although the
weakly damped system considered in \cite{FJG09} is slightly
different from ours in the formulation, by a careful calculation,
one can obtain the same \emph{a priori} estimates without those
restrictions therein. (ii) The equation for $\phi$ is a
Schr\"{o}dinger type equation, which does not enjoy the smoothing
property like parabolic equations. To show the precompactness of
$\phi$, we use a suitable decomposition to split the trajectory into
two parts: one decays exponentially fast to zero, and the other one
satisfies a certain compactness property. We recall that in
\cite{FJG09}, no results on the compactness of weak solutions were
obtained and only the existence of a weakly compact attractor was
proved. (iii) We prove the finite dimensionality (in terms of
fractal dimension) of the global attractor and the existence of an
exponential attractor. Although the global attractor represents the
first important step in the understanding of long-time dynamics of a
given evolutionary problem, it may also present some severe
drawbacks. Indeed, as simple examples show, the rate of convergence
to the global attractor may be arbitrarily slow. This fact makes the
global attractor very sensitive to perturbations and to numerical
approximation. In addition, it is usually very difficult to estimate
the rate of convergence to the global attractor and to express it in
terms of the physical parameters of the system. The concept of
exponential attractor has then been proposed in \cite{EFNT} to
possibly overcome these drawbacks. The exponential attractors
contain the global attractor, are finite dimensional, and attract
the trajectories exponentially fast. Comparing with the global
attractor, an exponential attractor turns out to be much more robust
to perturbations. Besides, it provides a way of proving that the
global attractor has finite fractal dimension. We refer to
\cite{MZsur} for a survey. In this paper, we apply a simple method
that also works in Banach spaces, due to \cite{EMZ} (see
\cite{GGMP,BG,EMZ1} for generalizations) to prove the existence of
an exponential attractor. As a byproduct, we obtain the finite
fractal dimensionality of the global attractor.

The remaining part of this paper is organized as follows.
 In Section 2, we prove the existence and uniqueness of global weak solutions to
 problem \eqref{3}-\eqref{4}. In Section
3, we show that problem \eqref{3}-\eqref{4}
 possesses a compact global attractor $\mathcal{A}$ in $\cH^1_0\times
 \cH^1_0$. In the last Section 4,
 we prove the existence of an exponential attractor $\mathcal{E}$, whose basin of attraction is the whole space
 $(\cH^2\cap\cH^1_0)\times
 (\cH^2\cap\cH^1_0)$.

\section{Global Existence and Uniqueness of Weak Solutions}
\noindent \setcounter{equation}{0} In order to prove the existence
of weak solutions to problem \eqref{3}-\eqref{4}, we shall use the
Faedo-Galerkin method to find approximate solutions. After deriving
some uniform \emph{a priori} estimates for the approximate solution,
we can pass to the limit. We denote by $C$ and $C_i$ positive
constants that may vary from place to place. Special dependence will
be indicated if it is necessary.

\begin{theorem}\label{T2.1}
 Suppose that assumptions (A1)-(A3) are satisfied.
 For any $(v_0,\phi_0)\in \cH^1_0(\Omega)\times \cH^1_0(\Omega)$
 and $T>0$, the initial boundary value problem \eqref{3}-\eqref{4}
 admits a global weak solution $(v,\phi)$.
\end{theorem}
\begin{proof}
 \emph{Step 1. Galerkin's approximation}

Let $\{\omega_j\}, j=1,2,...$ be a system of eigenfunctions of the
operator $A$, that is,
 \beq -\Delta \omega_j=\lambda_j\omega_j,
\quad \text{in}\  \Omega, \quad \text{and}\ \  \omega_j=0,\quad
\text{on} \ \Gamma,
 \eeq
 where $0<\lambda_1 \leq \lambda_2\leq ...$ are the eigenvalues.
 It is easy to see that $\{w_j\}$ forms base functions of
 $H^1_0(\Omega)$ as well as $L^2(\Omega)$. Moreover, $\omega_j\in C^\infty$, $j\in \mathbb{N}$.

Let $l$ be a given positive integer. We denote the approximate
solutions of problem \eqref{w1}-\eqref{w2} by $v_l(x,t)$ and
$\phi_l(x,t)$ such that
 $ v_l(x,t)=\sum_{j=1}^l\alpha_{jl}(t)\omega_j(x)$, $
 \phi_l(x,t)=\sum_{j=1}^l\beta_{jl}(t)\omega_j(x)$,
 where $\alpha_{jl}(t)$, $\beta_{jl}(t)$, $(j=1,2,...,l)$ are complex-valued
 functions that satisfy the following system of ordinary differential equations of first
  order: for $j=1,2,...,l$,
 \beq
  d( v_{lt},\omega_j)-i\left(a-\frac{1}{U}\right)(
v_l,\omega_j)-\frac{i g}{U}(\phi_l,\omega_j)-\frac{ic}{4m}(\Delta
v_l,\omega_j)+i b(|v_l|^2v_l,\omega_j)=( f, \omega_j),\label{x1}
 \eeq
 \beq
(\phi_{lt},\omega_j)-\frac{i g}{U}(
v_l,\omega_j)+i\left(\frac{g^2}{U}+2\nu-2\mu\right)(\phi_l,\omega_j)-\frac{i}{4m}(\Delta\phi_l,\omega_j)+\gamma(\phi_l,\omega_j)=(
h,\omega_j),\label{x2}
 \eeq
 with the initial data
 \beq (v_l(0), \omega_j)=\eta_{jl},\quad (\phi_l(0), \omega_j)=\zeta_{jl}. \label{x3}
 \eeq
$\eta_{jl}, \zeta_{jl}$ are constants such that as $l\to +\infty$
 \beq \sum_{j=1}^l \eta_{jl} \omega_j \to v_0, \quad \sum_{j=1}^l\zeta_{jl}\omega_j\to \phi_0,\quad \text{strongly in}
  \ \cH^1_0(\Omega). \nonumber
 \eeq
 Existence of such $\eta_{jl}, \zeta_{jl}$ follows from the fact
 that $(v_0, \phi_0)\in \cH_0^1(\Omega)\times \cH_0^1(\Omega)$ and the
 definition of $\{\omega_j\}$. Actually, we can just take $\eta_{jl}=(v_0,
 \omega_j)$, $\zeta_{jl}=(\phi_0, \omega_j)$.

 The standard theory for nonlinear ordinary differential equations of
first order (i.e., the Picard iteration method) ensures  that for
each $l$, the initial value problem \eqref{x1}-\eqref{x3} admits a
unique local solution $(v_l, \phi_l)$ on $[0,t_0]$ where $t_0$
depends only on $|\xi_{jl}|$ and $|\zeta_{jl}|$. We omit the details
here.

\emph{Step 2. a priori estimates}

We now try to obtain some \emph{a priori} estimates for the
approximate solutions. Multiplying \eqref{x1} by
$\overline{\alpha_{jl}}(t)$, $\frac{d}{dt}\overline{\alpha_{jl}}(t)$
and $\lambda_j\overline{\alpha_{jl}}(t)$, respectively, summing over
$j$ from $1$ to $l$ and taking the imaginary part of the results, we
get
 \begin{eqnarray}
&&
\frac{d_i}{2}\frac{d}{dt}\|v_l\|^2+\left(\frac{1}{U}-a\right)\|v_l\|^2+\frac{c}{4m}\|\nabla
v_l\|^2+b\|v_l\|^4_{\cL^4}\nonumber\\
& =&\mathrm{Im}\int_\Omega f\overline{v_l}
dx+\frac{g}{U}\mathrm{Re}\int_\Omega \phi_l
\overline{v_l}dx-d_r\mathrm{Im}\int_\Omega v_{lt}\overline{v_l}dx\nonumber\\
&\leq&
\frac{d_i}{4}\|v_{lt}\|^2+\|\phi_l\|^2+\|f\|^2+C\|v_l\|^2.\label{ap1}
 \end{eqnarray}
 \begin{eqnarray}
&& \frac{1}{2}\left(\frac{1}{U}-a\right)\frac{d}{dt}\|v_l\|^2
+\frac{c}{8m}\frac{d}{dt}\|\nabla
v_l\|^2+\frac{b}{4}\frac{d}{dt}\|v_l\|^4_{\cL^4}+d_i\|v_{lt}\|^2\nonumber
\\
&=&\mathrm{Im}\int_\Omega f\overline{v_{lt}}
dx+\frac{g}{U}\mathrm{Re}\int_\Omega \phi_l \overline{v_{lt}} dx\leq
\frac{d_i}{2}\|v_{lt}\|^2+C(\|f\|^2+\|\phi_l\|^2).\label{ap2}
 \end{eqnarray}
\begin{eqnarray}
&& \frac{d_i}{2}\frac{d}{dt}\|\nabla
v_l\|^2+\left(\frac{1}{U}-a\right)\|\nabla
v_l\|^2+\frac{c}{4m}\|\Delta v_l\|^2+2b\int_\Omega |\nabla
v_l|^2|v_l|^2 dx
\nonumber\\
&=&-\mathrm{Im}\int_\Omega  f\Delta\overline{v_l}dx
-\frac{g}{U}\mathrm{Re}\int_\Omega \phi_l\Delta\overline{v_l}dx
-b\mathrm{Re}
\int_\Omega \nabla\overline{v_l}\cdot \nabla\overline{v_l}v_l^2dx +d_r\mathrm{Im} \int_\Omega v_{lt}
\Delta \overline v_l dx\nonumber\\
&\leq& b\int_\Omega |\nabla v_l|^2|v_l|^2dx +\frac{c}{8m}\|\Delta
v_l\|^2+C(\|f\|^2+\|\phi_l\|^2)+C_1\|v_{lt}\|^2.\label{ap3}
\end{eqnarray}
 On the other
hand, multiplying \eqref{x2} by $\overline{\beta_{jl}}(t)$,
$\lambda_j\overline{\beta_{jl}}$ and
$\lambda_j^{-1}\frac{d}{dt}\overline{\beta_{jl}}(t)$, respectively,
summing over $j$ form $1$ to $l$, and taking the real part, we
obtain
 \beq
 \frac{1}{2}\frac{d}{dt}\|\phi_l\|^2+\gamma\|\phi_l\|^2
 =\mathrm{Re}\int_{\Omega}h\overline{\phi_l}dx-\frac{g}{U}\mathrm{Im}\int_{\Omega}v\overline{\phi_l}dx
 \leq\frac{\gamma}{2}\|\phi_l\|^2+C(\|h\|^2+\|v_l\|^2).\label{ap1a}
 \eeq
 \begin{eqnarray}
 \frac{1}{2}\frac{d}{dt}\|\nabla\phi_l\|^2+\gamma\|\nabla\phi_l\|^2&=&\mathrm{Re}
 \int_{\Omega}\nabla h\cdot\nabla\overline{\phi_l}dx-\frac{g}{U}\mathrm{Im}\int_{\Omega}\nabla v_l\cdot\nabla\overline{\phi_l}dx
 \nonumber\\
&\leq&\frac{\gamma}{2}\|\nabla \phi_l\|^2+C(\|\nabla h\|^2+\|\nabla
v_l\|^2).\label{ap2a}
 \end{eqnarray}
 \begin{eqnarray}
 \frac{\gamma}{2}\frac{d}{dt}\|\phi_l\|^2_{\cH^{-1}} + \|\phi_{lt}\|^2_{\cH^{-1}}
 &=&\left(\frac{g^2}{U}+2\nu-2\mu\right)\mathrm{Im}\int_\Omega \phi_l\Delta^{-1}\overline{\phi_{lt}}dx -\frac{1}{4m}\mathrm{Im}\int_\Omega
 \phi_l
 \overline{\phi_{lt}}dx  \nonumber\\
 && -\frac{g}{U}\mathrm{Im}\int_\Omega v_l\Delta^{-1}
 \overline{\phi_{lt}}dx +\mathrm{Re}\int_\Omega h\Delta^{-1}\overline{\phi_{lt}}dx\nonumber\\
 &\leq&\frac{1}{2}\|\phi_{lt}\|_{\cH^{-1}}^2+C(\|\phi_l\|_{\cH^{1}}^2+\|v_l\|_{\cH^{-1}}^2+\|h\|_{\cH^{-1}}^2).\label{ap3a}
 \end{eqnarray}
 Multiplying \eqref{ap3} by a small positive constant $\kappa\in (0, \frac{d_i}{8C_1})$,
 adding it with \eqref{ap1}, \eqref{ap2}, \eqref{ap1a}, \eqref{ap2a} together, we obtain
 \begin{equation}
 \frac{d}{dt}\Upsilon_1(t)+\Upsilon_2(t)\leq
 C(\|v_l(t)\|_{\cH^1}^2+\|\phi_l(t)\|^2+\|f\|^2+\|h\|_{\cH^1}^2),\label{gron1}
 \end{equation}
 where
 \begin{eqnarray}
 \Upsilon_1(t)&=& \left(\frac{c}{8m}+\frac{\kappa d_i}{2}\right)\|\nabla
 v_l\|^2+\frac12\left(d_i+\frac1U-a\right)\|v_l\|^2+\frac{b}{4}\|v_l\|_{\cL^4}^4+\frac12 \|\phi_l\|^2_{\cH^1},\label{U1}\\
 \Upsilon_2(t)&=& \frac{\kappa c}{8m} \|\Delta v_l\|^2+
 \left[\frac{c}{4m}+\kappa\left(\frac1U-a\right)\right]\|\nabla v_l\|^2+ b\|v_l\|_{\cL^4}^4+\kappa b\int_\Omega |\nabla
v_l|^2|v_l|^2 dx\nonumber\\
 && +\left(\frac{d_i}{4}-\kappa C_1\right)\|v_{lt}\|^2+\left(\frac{1}{U}-a\right)\|v_l\|^2+
 \frac{\gamma}{2}\|\phi_l\|_{\cH^1}^2.\label{U2}
 \end{eqnarray}
We infer from the assumptions (A2), (A3) and the condition on $\kappa$ that the coefficients of
all the terms in \eqref{U1} and \eqref{U2} are positive. Then it
follows from \eqref{gron1} that
 \beq \frac{d}{dt}\Upsilon_1(t)\leq
 C_2\Upsilon_1(t)+C_3(\|f\|^2+\|h\|_{\cH^1}^2).\nonumber
 \eeq
 By the Gronwall inequality and assumption
 (A1), we conclude that for  arbitrary  $T>0$:
 \beq
 \Upsilon_1(t) \leq e^{C_2t}\left[\Upsilon_1(0)+\frac{C_3}{C_2}(\|f\|^2+\|h\|_{\cH^1}^2)\right] , \quad \forall t\in [0,T].
 \eeq
 As  a  result,
 \beq
 \|v_l(t)\|^2_{\cH^1}+\|\phi_l(t)\|^2_{\cH^1}\leq C_T, \quad \forall t\in
 [0,T],\label{es1}
 \eeq
 where $C_T$ is a constant depending on
 $\|v_0\|_{\cH^1}$, $\|\phi_0\|_{\cH^1}$, $\|f\|$, $\|h\|_{\cH^1}$, $T$, $\Omega$,
 and the coefficients of the system. Turning back to \eqref{gron1} and integrating with respect to
 time, we can see that
 \beq
 \int_0^T \Upsilon_2(t) dt \leq \Upsilon_1(0)+C_2\int_0^T \Upsilon_1(t)
 dt +C_3T(\|f\|^2+\|h\|_{\cH^1}^2)\leq C_T,
 \eeq
 which implies that
 \beq
 \int_0^T (\|v_l(t)\|_{\cH^2}^2+\|v_{lt}\|^2+\|\phi_l(t)\|_{\cH^1}^2 )dt\leq
 C_T.\label{es2}
 \eeq
 Finally, we infer from \eqref{ap3a}, \eqref{es1} and assumptions (A1), (A2) that
 \beq
 \int_0^T \|\phi_{lt}(t)\|_{\cH^{-1}}^2dt\leq C_T.\label{es3}
 \eeq
 The above uniform estimates imply that the solution
 $(\alpha_{1l}(t),...,\alpha_{ll}(t),\beta_{1l}(t),...,\beta_{ll}(t))$ to ODE problem
  \eqref{x1}-\eqref{x3} can be extended to $[0, T]$,
for any $T>0$. Moreover, on $[0,T]$ we have the following uniform
\emph{a priori} estimates:
  \beq
  \begin{cases}
  \label{b1}
  v_l, \ \phi_l \ \ \text{    uniformly bounded in  } L^{\infty}((0,T),\cH^1_0),\\
  v_l \qquad\text{       uniformly bounded in    } L^{2}((0,T),\cH^2),\\
  \phi_l \qquad\text{       uniformly bounded in    } L^{2}((0,T),\cH^1_0),\\
  v_{lt} \qquad\text{       uniformly bounded in    } L^{2}((0,T),\cL^2),\\
  \phi_{lt} \qquad\text{       uniformly bounded in    } L^{2}((0,T),\cH^{-1}),\\
  \end{cases}
  \eeq
  \emph{Step 3. Convergence of the approximate solutions as $l\to +\infty$}

  The uniform bounds \eqref{b1} yield that there exist functions $(v, \phi)$ and subsequences of $\{v_l\}$ and $\{\phi_l\}$ (still
denoted by $\{v_l\}$ and $\{\phi_l\}$ for the sake of simplicity) such that as $l\to +\infty$,
 \beq
 \label{conve}
 \begin{cases}
 v_l\rightarrow v,\ \ \phi_l\rightarrow\phi,  \qquad\text{weakly-* in }\quad L^{\infty}((0,T),\cH^1_0),\\
 v_l\rightarrow v \quad\text{weakly in }\quad L^{2}((0,T),\cH^2),\\
 \phi_l\rightarrow \phi \quad\text{weakly in }\quad L^{2}((0,T),\cH^1_0),\\
 v_{lt}\rightarrow v_t \quad\text{weakly in  }\quad L^{2}((0,T),\cL^2),\\
 \phi_{lt}\rightarrow \phi_t \quad\text{weakly in   }\quad L^{2}((0,T),\cH^{-1}).
 \end{cases}
  \eeq
 From $\phi\in L^2((0,T),\cH_0^1), \phi_t\in L^{2}((0,T),\cH^{-1})$ and \cite[Lemma II.3.2]{Temam} we know that
 $\phi\in C([0,T],\cL^2)$. Besides, by the following result (cf. e.g., \cite{LM})
 \begin{lemma}\label{em}
 Let $X \subset Y $ be two Hilbert spaces, and suppose that the embedding of $X$ into $Y$ is
compact.
The following continuous embedding holds:
 $\left\{ f\in L^2((0,T), X),\ f_t\in L^2((0,T), Y)\right\}$ $ \hookrightarrow C([0,T];
 [X,Y]_\frac12).$
 \end{lemma}
 \noindent and the fact that  $\cH_0^1=[\cH^2\cap \cH^1_0, \cL^2]_{\frac12}$ (cf. \cite{Temam}), we have (up to a subsequence)
 \beq
  v_l\to v\quad \text{ weakly in}\ C([0,T],\cH_0^1).\label{conve1}
 \eeq
  We infer from \cite[Lemma II.3.3]{Temam} that
  $\phi$ is weakly continuous with values in $\cH^1_0$.
  Namely, for any $\psi\in \cH^1_0$, $t\mapsto \int_\Omega \nabla \phi(t)\cdot\nabla \overline{\psi} dx$ is continuous.
  Arguing as in \cite{Temam}, we can get an equality similar to \eqref{ap2a} which holds in the distributional sense on $(0,T)$:
   \beq
 \frac{1}{2}\frac{d}{dt}\|\nabla\phi\|^2+\gamma\|\nabla\phi\|^2=\mathrm{Re}
 \int_{\Omega}\nabla h\cdot\nabla\overline{\phi}dx-\frac{g}{U}\mathrm{Im}\int_{\Omega}\nabla v\cdot\nabla\overline{\phi}dx.
 \eeq
 As a result,  $t\mapsto \|\nabla \phi(t)\|^2$ is also continuous on $[0,T]$. Since $\|\nabla\cdot\|$
 is the equivalent norm on $\cH_0^1$, we conclude that $\phi\in C([0,T],\cH^1_0)$.

 The well-known Aubin-Lions lemma implies that there is a subsequence
of $v_l$, still denoted by $v_l$ such that
 \beq v_l\rightarrow v
\quad\text{strongly in}
 \quad L^2((0,T),\cH^1_0).
 \eeq
 Hence, there is a
subsequence of $v_l$, still denoted by $v_l$ such that $v_l$ almost
everywhere converges to $v$ in $Q_T=\Omega\times[0,T]$. It turns out
that $|v_l|^2v_l$ almost everywhere converges to $|v|^2v$ in $Q_T$.
On the other hand, it follows from \eqref{b1} that $|v_l|^2v_l$ is
uniformly bounded in
 $L^{\infty}((0,T),\cL^2)$ and hence in $ L^{2}((0,T),\cL^2)$.
 Therefore, we infer that the weak limit of $|v_l|^2v_l$ in
 $L^2([0,T],\cL^2(\Omega))$ equals to $|v|^2v$:
  \beq
|v_l|^2v_l\rightarrow |v|^2v\quad\text{weakly in }\quad L^2((0,T),
\cL^2).
 \eeq
 Passing to the limit $l\to +\infty$, we can infer from the above convergence properties of $v_l, \phi_l$ that
\eqref{w1} and \eqref{w2} are satisfied.  Concerning the initial
data, we infer from \eqref{conve} that (cf. e.g.,
 \cite[Lemma 3.1.7]{Z02}) that
 \beq
 \begin{cases}
 v_l(0)=\sum_{j=1}^l\eta_{jl}\omega_j\rightarrow v(0)\quad
 \text{weakly in} \quad \cL^2(\Omega), \\
 \phi_l(0)=\sum_{j=1}^l\zeta_{jl}\omega_j\rightarrow \phi(0)\quad
 \text{weakly in } \quad \cH^{-1}(\Omega).
 \end{cases}
 \eeq
 On the other hand, we know that $(v_l(0),\phi_l(0))$ strongly converges in $\cH^1_0\times \cH^1_0$;
 hence, it also weakly converges to $(v_0,\phi_0)$ in $\cL^2\times
\cH^{-1}$. By the uniqueness of the limit, we have $ v(0)=v_0$,
$\phi(0)=\phi_0$.

Summing up, we have proved the existence of a global weak solution
$(v,\phi)$ to problem \eqref{3}-\eqref{4}. The proof is complete.
\end{proof}

Next, we show the continuous dependence result on the initial data
that yields the uniqueness of weak
 solutions to problem \eqref{3}-\eqref{4}:
 \begin{theorem}\label{cod}
 For any $(v_{01}, \phi_{01}), (v_{02}, \phi_{02})\in \cH^1_0(\Omega)\times \cH^1_0(\Omega)$, we denote
 the corresponding global weak solutions to
 problem \eqref{3}-\eqref{4} by $(v_1,\phi_1)$ and
 $(v_2,\phi_2)$, respectively. For any $T>0$, it holds
 \begin{eqnarray}
 && \|v_1(t)-v_2(t)\|_{\cH^1}^2+\|\phi_1(t)-\phi_2(t)\|_{\cH^1}^2+ \int_0^t\|v_{1t}(t)-v_{2t}(t)\|^2dt\nonumber\\
  &\leq& L_1e^{L_2t}(\|v_{01}-v_{02}\|_{\cH^1}^2+\|\phi_{01}-\phi_{02}\|_{\cH^1}^2),\quad \forall\ 0\leq t\leq T,\label{contii}
 \end{eqnarray}
 where $L_1, L_2$ are positive constants depending on $\|v_{01}\|_{\cH^1},  \|\phi_{01}\|_{\cH^1}$,
 $\|v_{02}\|_{\cH^1}, \|\phi_{02}\|_{\cH^1}$, $|\Omega|$, $f$, $h$ and coefficients of system \eqref{3}.
 \end{theorem}
 \begin{proof}
 We shall just perform formal computations
that can be justified within the same Galerkin scheme used above.
Let $v=v_1-v_2$, $\phi=\phi_1-\phi_2$, $ v(0)=v_{01}-v_{02}$,
$\phi(0)=\phi_{01}-\phi_{02}$. Then the differences $(v,\phi)$
satisfy a.e. in $[0,T]$ that
 \beq
  \label{u1} d v_t +i\left(\frac{1}{U}-a\right)v -\frac{i
 g}{U} \phi-\frac{ic}{4m}\Delta v+ib(|v_1|^2v_1-|v_2|^2v_2)=0,
 \eeq
 \beq \label{u2}
 \phi_t-\frac{i g}{U}v+i\left(\frac{g^2}{U}+2\nu-2\mu\right)\phi-\frac{i}{4m}\Delta\phi +\gamma\phi=0.
 \eeq
 Multiplying \eqref{u1} by $\overline{v}_t$, integrating over $\Omega$ and taking the imaginary part of the result, we have
 \begin{eqnarray}
 &&\frac{d}{dt}\left[\frac{1}{2}\left(\frac{1}{U}-a\right)\|v\|^2+\frac{c}{8m}\|\nabla v\|^2\right]+ d_i\|v_t\|^2\nonumber\\
 &=&\frac{g}{U}\mathrm{Re}\int_{\Omega}\phi\overline{v}_tdx-b\mathrm{Re}\int_\Omega(|v_1|^2v_1-|v_2|^2v_2)\overline{v}_tdx\nonumber\\
 &\leq&\frac{d_i}{2}\|v_t\|^2+C\|\phi\|^2+C\||v_1|^2v_1-|v_2|^2v_2\|^2\nonumber\\
 &\leq&\frac{d_i}{2}\|v_t\|^2+C\|\phi\|^2+C\left[\|v_1\|_{\cL^6}^2\|v\|_{\cL^6}+ \|v_2\|_{\cL^6}(\|v_1\|_{\cL^6}
 +\|v_2\|_{\cL^6})\|v\|_{\cL^6}\right]^2\nonumber\\
 &\leq& \frac{d_i}{2}\|v_t\|^2+C(\|\phi\|^2+\|v\|_{\cH^1}^2).\label{qi1}
\end{eqnarray}
In above, we have used the uniform-in-time estimate \eqref{uniH1}
instead of \eqref{es1}.  Multiplying \eqref{u2} by
$\overline{\phi}-\Delta \overline{\phi}$, integrating over $\Omega$
and taking the real part, we get
 \begin{eqnarray}
 && \frac{1}{2}\frac{d}{dt}(\|\phi\|^2+\|\nabla\phi\|^2)+\gamma(\|\phi\|^2+\|\nabla\phi\|^2)
 =-\frac{g}{U}\mathrm{Im}\int_{\Omega}(v\overline{\phi}+\nabla
 v\cdot\nabla\overline{\phi})dx\nonumber\\
 &\leq&\frac{\gamma}{2}(\|\phi\|^2+\|\nabla\phi\|^2)+C(\|v\|^2+\|\nabla v\|^2).\label{qi2}
 \end{eqnarray}
 Adding the above estimates together, we have
 \begin{eqnarray}
 &&\frac{d}{dt}\left[\frac{1}{2}\left(\frac{1}{U}-a\right)\|v\|^2+\frac{c}{8m}\|\nabla v\|^2+\frac12\|\phi\|_{\cH^1}^2\right]
 + \frac{d_i}{2}\|v_t\|^2+\frac{\gamma}{2}\|\phi\|_{\cH^1}^2\nonumber\\
 &\leq& C(\|\phi\|_{\cH^1}^2+\|v\|_{\cH^1}^2).\label{ggg}
 \end{eqnarray}
 Then our conclusion \eqref{contii} easily follows from \eqref{ggg} and the standard Gronwall lemma. The proof is complete.
\end{proof}

\begin{corollary}\label{uniw}
Under the assumptions of Theorem \ref{T2.1}, the global weak
solution $(v,\phi)$ to problem \eqref{3}-\eqref{4} is unique.
\end{corollary}

The above results imply that the unique global weak solution to
problem \eqref{3}-\eqref{4} defines a strongly continuous nonlinear
semigroup $S(t)$ acting on $\cH_0^1(\Omega)\times \cH_0^1(\Omega)$,
such that $ (v(t),\phi(t))=S(t)(v_0,\phi_0). $

\section{Existence of the Global Attractor}
\setcounter{equation}{0}

 In this section, we study the existence of a global
attractor to problem \eqref{3}-\eqref{4}.
 For this purpose, we will show the existence of an absorbing
set and some precompactness of the weak solution  $(v, \phi)$. In
the remaining part of the paper, we shall exploit some formal
\emph{a priori} estimates, which can be justified rigorously by the
approximate procedure in the previous section and the standard dense
argument.

\begin{proposition}
 \label{P3.1}
 Let assumptions (A1)-(A3) be satisfied.
There exists a positive constant $R_0$ such that the ball
$$\mathcal{B}_0=\{(v,\phi)\in \cH_0^1(\Omega)\times \cH_0^1(\Omega)\ |\ \|v\|^2_{\cH^1}+\|\phi\|^2_{\cH^1}\leq
R_0\}$$ is a bounded absorbing set for the dynamical system $S(t)$
associated with problem \eqref{3}-\eqref{4}. Namely, for any bounded
set $\mathcal{B}\subset \cH_0^1(\Omega)\times \cH_0^1(\Omega)$,
there is $t_0 =t_0(\mathcal{B})$ such that $S(t)\mathcal{B} \subset
\mathcal{B}_0$ for every $t \geq t_0$.
 \end{proposition}

\begin{proof}
Within the proof, we denote by $C_j\  (j=1,2,...)$ positive
constants that may depend on the coefficients of the system
\eqref{3}, $\Omega$, but not on the initial data $v_0, \phi_0$ and
time. Multiplying the first equation in \eqref{3} by $\overline v$
and $\overline v_t$, respectively, integrating over $\Omega$ and
taking the imaginary part of the results, we have
 \begin{eqnarray}
&& \frac{d_i}{2}\frac{d}{dt}\|v\|^2+\left(\frac{1}{U}-a\right)\|v\|^2+\frac{c}{4m}\|\nabla
v\|^2+ b\int_{\Omega}|v|^4dx
\nonumber\\
&=& \frac{g}{U}\mathrm{Re}\int_{\Omega}\phi\overline{v}dx-d_r
\mathrm{Im}\int_{\Omega}v_t\overline{v}dx+\mathrm{Im}\int_{\Omega}f\overline{v}dx \nonumber\\
&\leq&
\frac{1}{2}\left(\frac{1}{U}-a\right)\|v\|^2+C_1(\|\phi\|^2+\|v_t\|^2+\|f\|^2),\label{abo1}
 \end{eqnarray}
 \begin{eqnarray}
&& \frac{d}{dt}\left[\frac{1}{2}\left(\frac{1}{U}-a\right)\|v\|^2+\frac{c}{8m}\|\nabla v\|^2
+\frac{b}{4}\int_{\Omega}|v|^4dx\right]+d_i\|v_t\|^2\nonumber\\
&=&\frac{g}{U}\mathrm{Re}\int_{\Omega}\phi\overline{v}_tdx+\mathrm{Im}\int_{\Omega}f\overline{v}_tdx\nonumber\\
&\leq& \frac{d_i}{2}\|v_t\|^2+C_2(\|\phi\|^2+\|f\|^2).\label{abo2}
 \end{eqnarray}
Multiplying the second equation in \eqref{3} by $\overline \phi$ and
$-\Delta \overline\phi$, respectively, integrating over $\Omega$ and
taking the real part, we get
 \begin{eqnarray}
\frac{1}{2}\frac{d}{dt}\|\phi\|^2+\gamma\|\phi\|^2&=&
\mathrm{Re}\int_{\Omega}h\overline{\phi}dx-\frac{g}{U}\mathrm{Im}\int_{\Omega}v\overline{\phi}dx\nonumber\\
&\leq& \frac{\gamma}{2}\|\phi\|^2+C_3(\|v\|^2+\|h\|^2),\label{abo3}
 \end{eqnarray}
 \begin{eqnarray}
 \frac{1}{2}\frac{d}{dt}\|\nabla\phi\|^2+\gamma\|\nabla\phi\|^2
 &=&\mathrm{Re}\int_{\Omega}\nabla h\cdot\nabla\overline{\phi}dx
 -\frac{g}{U}\mathrm{Im}\int_{\Omega}\nabla v\cdot\nabla\overline{\phi}dx\nonumber\\
&\leq&\frac{\gamma}{2}\|\nabla \phi\|^2+C_4(\|\nabla v\|^2+\|\nabla
h\|^2).\label{abo4}
 \end{eqnarray}
Now multiplying \eqref{abo1} by $\kappa_1>0$, \eqref{abo3} by
$\kappa_2>0$ and \eqref{abo4} by  $\kappa_3>0$, adding together the
resulting inequalities with \eqref{abo2}, we obtain that
 \begin{eqnarray}
 && \frac{d}{dt}\left[\frac12\left(\kappa_1d_i+\frac1U-a\right)\|v\|^2+\frac{c}{8m}\|\nabla
 v\|^2+\frac{b}{4}\int_\Omega|v|^4dx+\frac{\kappa_2}{2}\|\phi\|^2+\frac{\kappa_3}{2}\|\nabla\phi\|^2\right]\nonumber\\
 && +
 \frac{\kappa_1}{2}\left(\frac{1}{U}-a\right)\|v\|^2+\left(\frac{c\kappa_1}{4m}-C_4\kappa_3\right)\|\nabla
 v\|^2+\kappa_1b\int_\Omega |v|^4dx\nonumber\\
 &&
 +\left(\frac{\gamma\kappa_2}{2}-C_1\kappa_1-C_2\right)\|\phi\|^2
 +\frac{\gamma\kappa_3}{2}\|\nabla \phi\|^2+\left(\frac{d_i}{2}-C_1\kappa_1\right)\|v_t\|^2\nonumber\\
&\leq& C_3\kappa_2\|v\|^2+(C_1\kappa_1+C_2)\|f\|^2+
C_3\kappa_2\|h\|^2+C_4\kappa_3\|\nabla h\|^2.\label{ABO}
 \end{eqnarray}
By the Young inequality, we have for some $\kappa_4>0$,
 \beq \|v\|^2 \leq
 \kappa_4\|v\|_{\cL^4}^4+\frac{|\Omega|}{4\kappa_4}.\label{Yo}
 \eeq
Take
 \beq
 \kappa_1
 =\frac{d_i}{4C_1},\quad \kappa_2=\frac{4C_2+d_i}{\gamma},\quad \kappa_3=\frac{cd_i}{32mC_1C_4},\quad \kappa_4=\frac{\gamma b
 d_i}{8C_1C_3(4C_2+d_i)}.
 \eeq
 We infer from \eqref{ABO} that the following inequality holds
 \beq
  \frac{d}{dt}E_1(t)+ C_5 E_1(t)+\frac{d_i}{4}\|v_t\|^2\leq C_6,\label{ABO1}
 \eeq
 where
\beq
E_1(t)=\frac12\left(\kappa_1d_i+\frac1U-a\right)\|v(t)\|^2+\frac{c}{8m}\|\nabla
 v(t)\|^2+\frac{b}{4}\int_\Omega|v(t)|^4dx+\frac{\kappa_2}{2}\|\phi(t)\|^2+\frac{\kappa_3}{2}\|\nabla\phi(t)\|^2.\nonumber
 \eeq
 Then \eqref{ABO1} yields that
 \beq  E_1(t)\leq e^{-C_5t}E_1(0)+\frac{C_6}{C_5},\quad \forall t\geq 0.\label{ABO2}\eeq
On the other hand, let
$$ E_2(t)=\|v(t)\|_{\cH^1}^2+\|v(t)\|_{\cL^4}^4+\|\phi(t)\|_{\cH^1}^2.$$
 It is easy to see that there exist $C_7, C_8,C_9>0$ such that for
 all $t\geq 0$,
 \beq
 \begin{cases} C_7 E_2(t)\leq E_1(t)\leq C_8E_2(t),\\
 \|v(t)\|_{\cH^1}^2+\|\phi(t)\|_{\cH^1}^2\leq E_2(t)\leq
 \|v(t)\|_{\cH^1}^2+C_9\|v(t)\|_{\cH^1}^4+\|\phi(t)\|_{\cH^1}^2.
 \end{cases}
 \eeq
 This and \eqref{ABO2} imply that
 \beq \|v(t)\|_{\cH^1}^2+\|\phi(t)\|_{\cH^1}^2\leq E_2(t)\leq \frac{C_8}{C_7}
 e^{-C_5t} (\|v_0\|_{\cH^1}^2+\|v_0\|_{\cH^1}^4+\|\phi_0\|_{\cH^1}^2)+\frac{C_6}{C_5C_7},\quad \forall t\geq
 0.\label{uniH1}
 \eeq
 Finally, we can take $R_0=\frac{2C_6}{C_5C_7}$. The proof is complete.
 \end{proof}

 \begin{remark}\label{R3.1}
 Proposition \ref{P3.1} implies that the trajectories $(v(t), \phi(t))$ starting from any bounded set $\mathcal{B}$
 will eventually enter the ball $\mathcal{B}_0$ in $\cH_0^1\times \cH_0^1$ of radius
 $(R_0)^\frac12$ uniformly in time. Noticing that, $\mathcal{B}_0\subset \tilde{\mathcal{B}}_0:=\bigcup_{t\geq
0}S(t)\mathcal{B}_0$, we can see that $\tilde{\mathcal{B}}_0$ also
serves as an absorbing set of $S(t)$. Moreover,
$\tilde{\mathcal{B}}_0$ is invariant under $S(t)$ for $t\geq 0$.
 \end{remark}
%%%%%%%%%%%%%%%%%%%%%%%%%%%%%%%%%%%%%%%%%%%%%%%%%%%%%%%%%%%%%%%%%%%%%%%%%%
 Our next goal is to study the precompactness of the weak
solution $(v, \phi)$ of problem \eqref{3}-\eqref{4}.

\begin{lemma}
\label{P4.1}
 Under assumptions of Theorem \ref{T2.1}, the following
uniform estimate holds:
 \beq\|v(t)\|_{\cH^2(\Omega)}\leq C\left(1+\frac{1}{r}\right), \quad \forall \ t\geq r>0,
 \eeq
 where $C$ is a constant depending on $\|v_0\|_{\cH^1}$,
 $\|\phi_0\|_{\cH^1}$, $\Omega$, $f,h$ and the coefficients of the
 system \eqref{3}, but independent of $t$.
\end{lemma}
\begin{proof}
 For any $t\geq 0$ and $r>0$, integrating \eqref{ABO1} from $t$ to $t+r$, we
 infer from \eqref{uniH1} that
  \beq \label{int2}
   \int_t^{t+r}(\|v(\tau)\|_{\cH^1}^2+\|v(\tau)\|_{\cL^4}^4+\|\phi(\tau)\|_{\cH^1}^2+\|v_t(\tau)\|^2)d\tau\leq
   C.
   \eeq
Multiplying the first equation in \eqref{3} by $-\Delta \overline
v$, integrating over $\Omega$ and taking the imaginary part, we have
  \begin{eqnarray}
  && \frac{d_i}{2}\frac{d}{dt}\|\nabla v\|^2+
  \left(\frac{1}{U}-a\right)\|\nabla v\|^2+\frac{c}{4m}\|\Delta v\|^2+2b\int_\Omega|v|^2|\nabla v|^2dx\nonumber\\
  &=&-\frac{g}{U}\mathrm{Re}\int_\Omega\phi\Delta\overline{v}dx-b\mathrm{Re}\int_\Omega\nabla\overline{v}\cdot
  \nabla\overline{v}v^2dx+d_r\mathrm{Im}\int_\Omega v_t\Delta\overline{v}dx-\mathrm{Im}\int_\Omega f\Delta\overline{v}dx\nonumber\\
 &\leq&\frac{c}{8m}\|\Delta v\|^2+b\int_\Omega|v|^2|\nabla
 v|^2dx+C(\|v_t\|^2+\|\phi\|^2+\|f\|^2).
  \end{eqnarray}
  Integrating the above inequality from $t$ to $t+r$, we infer from  \eqref{int2} that
   \beq \int_t^{t+r}\|v(\tau)\|_{\cH^2}^2d\tau\leq C.\label{bdiH2}
   \eeq
   Next, multiplying the first equation in \eqref{3} by $-\Delta \overline
v_t$, integrating over $\Omega$ and taking the imaginary part, we
get
 \begin{eqnarray}
 && \frac{d}{dt}\left[\frac{1}{2}\left(\frac{1}{U}-a\right)\|\nabla v\|^2+\frac{c}{8m}\|\Delta v\|^2\right]+d_i\|\nabla v_t\|^2\nonumber\\
 &=&\frac{g}{U}\mathrm{Re}\int_\Omega\nabla \phi\nabla\overline{v}_tdx
 -b\mathrm{Re}\int_\Omega \nabla (|v|^2v)\cdot\nabla\overline{v}_tdx +\mathrm{Im}
 \int_\Omega \nabla f\cdot\nabla\overline{v}_tdx\nonumber \\
 &\leq&\frac{d_i}{2}\|\nabla v_t\|^2+C\left(\|\nabla \phi\|^2
 +\|\nabla f\|^2+\int_\Omega|v|^4|\nabla v|^2dx\right)\nonumber\\
 &\leq&\frac{d_i}{2}\|\nabla v_t\|^2+C(\|\nabla \phi\|^2+\|\nabla f\|^2+\|v\|^4_{\cL^{\infty}}\|v\|_{\cH^1}^2)\nonumber\\
 &\leq&\frac{d_i}{2}\|\nabla v_t\|^2+C(\|\nabla \phi\|^2+\|\nabla f\|^2+\|v\|^4_{\cH^1}\|\Delta
 v\|^2).\label{comv1}
 \end{eqnarray}
 In the last step, we use the three-dimensional Agmon
 inequality that for any $v\in \cH^2\cap \cH^1_0$, it holds $\|v\|^2_{\cL^{\infty}}\leq c(\Omega)\|\nabla v\|\|\Delta
 v\|.$ Then it easily follows from \eqref{comv1} that
  \beq
   \frac{d}{dt}y(t)\leq C h_1(t)y(t)+Ch_2(t),\nonumber
 \eeq
where
 $$ y(t)=\frac{1}{2}\left(\frac{1}{U}-a\right)\|\nabla
v(t)\|^2+\frac{c}{8m}\|\Delta v(t)\|^2, \quad
h_1(t)=\|v(t)\|_{\cH^1}^4, \quad
h_2(t)=\|\phi(t)\|_{\cH^1}^2+\|f\|_{\cH^1}^2.$$ Applying the
well-known uniform Gronwall lemma (cf. e.g., \cite[Lemma
III.1.1]{Temam}), we infer from \eqref{uniH1} and \eqref{bdiH2} that
for any $r>0$
 \beq
 y(t+r)\leq C\left(1+\frac{1}{r}\right), \quad \forall \ t\geq 0.
 \eeq
 The proof is complete.
\end{proof}

Since the continuous embedding $\cH^2\hookrightarrow \cH^1$ is
compact, Proposition \ref{P4.1} implies that $v(t)$ is precompact in
$\cH^1$ for $t\geq r$.

 Next, we prove the precompactness of
 $\phi(t)$. We note that $\phi$ satisfies a Schr\"{o}dinger type equation, which does
not enjoy the smoothing property like parabolic equations. To
overcome this difficulty, we shall decompose the solution $\phi$
into a uniformly stable part and a compact part such that
$$\phi=\phi^d+\phi^c,$$
where $\phi^d(t)$ and $\phi^c(t)$  satisfy the following systems
  \beq
  \begin{cases}
  \label{d}
  \phi^d_{t}+i\left(\frac{g^2}{U}+2\nu-2\mu\right)\phi^d-\frac{i}{4m}\Delta \phi^d+\gamma \phi^d=0,\\
  \phi^d|_{\Gamma}=0,\\
  \phi^d|_{t=0}=\phi_0,
  \end{cases}
  \eeq
  and
 \beq \label{c1}\begin{cases}
 \phi^c_{t}-\frac{ig}{U}v+i\left(\frac{g^2}{U}+2\nu-2\mu\right)\phi^c-\frac{i}{4m}\Delta \phi^c+\gamma \phi^c=h,\\
 \phi^c|_{\Gamma}=0,\\
 \phi^c|_{t=0}=0.
 \end{cases}
 \eeq
%%%%%%%%%%%%%%%%%
  \begin{lemma}\label{P4.2} Problem \eqref{d} admits a unique global
  weak solution $\phi^d(t)\in C([0,+\infty), \cH^1_0)$ and the following estimate holds:
 \beq
 \label{d2}\|\phi^d(t)\|_{\cH^1}= \|\phi_0\|_{\cH^1}e^{-\gamma t},\quad \forall t\geq 0.
 \eeq
 \end{lemma}
 \begin{proof}
The existence and uniqueness of solution $\phi^d$ to equation
\eqref{d} can be easily proven as in Section 2.
Multiplying
\eqref{d} by $\overline{\phi^d}-\Delta\overline{\phi^d}$,
integrating over $\Omega$ and taking the real part, we obtain
 \beq
  \frac{1}{2}\frac{d}{dt}(\|\phi^d\|^2+\|\nabla\phi^d\|^2)+\gamma(\|\phi^d\|^2+\|\nabla\phi^d\|^2)=
  0,\label{cccc}
 \eeq
 which easily yields \eqref{d2}.
 \end{proof}

 %%%%%%%%%%%%%%%%%%%%%%%%%%%%%%%%%%%%%%%%%%%%%%%%%

\begin{lemma}\label{P4.3}
For any $r>0$, it holds
 \beq
 \|\phi^c(t)\|_{\cH^2}\leq C,\quad t \geq r,\label{com4}
 \eeq
 where $K$ is a constant depending on
 $\|v_0\|_{\mathcal{H}^1_0}, \|\phi\|_{\mathcal{H}^1_0}$, $\Omega$, $f,h$, $r$ and the coefficients of system \eqref{3}.
\end{lemma}
\begin{proof} It follows from \eqref{uniH1} and \eqref{d2} that
 \beq
 \|\phi^c(t)\|_{\cH^1} \leq C,\quad \forall\  t\geq 0.\label{com5}
 \eeq
Differentiating \eqref{c1} with respect to $t$, multiplying the
resultant by $\overline{\phi^c_t}$, integrating over $\Omega$ and
taking the real part, we obtain
 \beq \frac{1}{2}\frac{d}{dt}\|\phi^c_t\|^2+\gamma\|\phi^c_t\|^2=-\frac{g}{U}\mathrm{Im}\int_{\Omega} v_t\overline{\phi^c}_tdx\leq\frac{\gamma}{2}\|\phi^c_t\|^2+\frac{g^2}{2\gamma U^2}\|v_t\|^2.
 \eeq
Namely,
 \beq
\frac{d}{dt}\|\phi^c_t\|^2+\gamma\|\phi^c_t\|^2\leq \frac{g^2}{\gamma U^2}\|v_t\|^2.\label{dphic}
 \eeq
It follows that for $\forall t\geq
r$,
 \begin{eqnarray}
 \|\phi^c_t(t)\|^2&\leq& e^{-\gamma t} \|\phi^c_t(0)\|^2
 +\frac{g^2}{\gamma U^2}e^{-\gamma t}\int_0^t e^{\gamma \tau }\|v_t(\tau)\|^2d\tau \nonumber\\
&\leq& Ce^{-\gamma r} (\|v_0\|^2+ \|h\|^2)+\frac{g^2}{\gamma U^2}e^{\gamma r}
\int_0^r \|v_t(\tau)\|^2d\tau+ \frac{g^2}{\gamma^2 U^2}(1-e^{-\gamma t})\sup_{\tau \geq r} \|v_t(\tau)\|^2\nonumber\\
&\leq& C+ \frac{g^2}{\gamma^2 U^2}\sup_{\tau \geq r} (\|v(\tau)\|_{\mathcal{H}^2}^2
+ \|\phi(\tau)\|^2+\|v(\tau)\|_{\mathcal{L}^6}^6)\nonumber\\
&\leq& C.\label{com3}
 \end{eqnarray}
Thus, we can deduce from the equation \eqref{c1} and Lemma \ref{P4.1} that
 \beq\|\phi^c(t)\|_{\cH^2}\leq C(\|v(t)\|+\|\phi^c_t(t)\|+\|\phi^c(t)\|+\|h\|)\leq C,\quad\forall t\geq r.
 \eeq
 The proof is complete.
\end{proof}
%%%%%%%%%%%%%%%%%%%%%%%%%%%%%%%%%
After the previous preparations, we are able to state the main result of this section:
 \begin{theorem}\label{GA}
Suppose that (A1)-(A3) are satisfied. The semigroup $S(t)$ defined
by the global weak solutions to problem \eqref{3}-\eqref{4} on
$\cH^1_0\times \cH^1_0$ possesses a compact connected global
attractor $\mathcal{A}\subset \cH^1_0\times \cH^1_0$, which is the
$\omega$-limit set of the absorbing set $\mathcal{B}_0$ such that
$\mathcal{A}=\omega(\mathcal{B}_0)$.
 \end{theorem}
%%%%%%%%%%%%%%%%%%%%%%%%%%%%%%%%%%%%%%%%%%%
\begin{proof}
Since $\mathcal{B}_0$ is a connected, invariant, bounded absorbing
set, our conclusion follows from Lemmas \ref{P4.1}-\ref{P4.3} and
the classical theory of dynamical systems (cf. e.g., \cite[Theorem
I.1.1]{Temam}).
\end{proof}
%%%%%%%%%%%%%%%%%%%%%%%%%%

\section{Existence of Exponential Attractors}
\setcounter{equation}{0}

The following proposition implies the dissipativity of the dynamical
system $S(t)$ when it is restricted to the regular space $(\cH^2\cap
\cH_0^1)\times(\cH^2\cap\cH^1_0)$.

\begin{proposition}\label{P5.3}
There exists $R_1\geq 0$ such that the ball
$$\mathcal{B}_1=\{(v,\phi)\in (\cH^2\cap \cH_0^1)\times(\cH^2\cap\cH^1_0)\ |\ \|v\|^2_{\cH^2}+\|\phi\|^2_{\cH^2}\leq
R_1\}$$ is a bounded absorbing set for $S(t)$ in $(\cH^2\cap \cH_0^1)\times(\cH^2\cap\cH^1_0).$
\end{proposition}
\begin{proof} Let $\mathcal{B}$ be any bounded set in $(\cH^2\cap
\cH_0^1)\times(\cH^2\cap\cH^1_0)$. In particular, there exist
$r_1\geq r_0\geq 0$ such that
 $$\sup\limits_{(v,\phi)\in \mathcal{B}}\|(v,\phi)\|_{\cH^1\times\cH^1}\leq r_0\qquad{and}\qquad\sup\limits_{(v,\phi)
 \in \mathcal{B}}\|(v,\phi)\|_{\cH^2\times\cH^2}\leq r_1.$$
 Within the proof, we denote by $K_j\  (j=1,2,...)$ positive constants that may
depend on the coefficients of the system \eqref{3}, $\Omega$, $f, h$, but not on the initial data $v_0,
\phi_0$ and time.

Differentiating the first equation in \eqref{3} with respect to
time, multiplying the result by $\overline v_t$, integrating over
$\Omega$ and taking the imaginary/real part, respectively, we have
 \begin{eqnarray}
 && \frac{d_i}{2}\frac{d}{dt}\|v_t\|^2+\frac{c}{4m}\|\nabla v_t\|^2+
 \frac{1-aU}{U}\|v_t\|^2\nonumber\\
 & =& \frac{g}{U}{\rm Re}\int_\Omega \phi_t \overline
v_t dx-b{\rm Re}\int_\Omega (|v|^2v)_t\overline v_t dx -d_r{\rm
Im}\int_\Omega v_{tt} \overline v_tdx,\label{vvt}
 \end{eqnarray}
 \beq
 \frac{d_r}{2}\frac{d}{dt}\|v_t\|^2-d_i{\rm Im}\int_\Omega v_{tt} \overline
 v_tdx=-\frac{g}{U}{\rm Im}\int_\Omega \phi_t \overline
v_t dx+b{\rm Im}\int_\Omega (|v|^2v)_t\overline v_t dx.\label{vvta}
 \eeq
 Inserting \eqref{vvta} into \eqref{vvt}, we have
  \begin{eqnarray}
 && \left(\frac{d_i}{2}+\frac{d_r^2}{2d_i}\right)\frac{d}{dt}\|v_t\|^2+\frac{c}{4m}\|\nabla v_t\|^2+
 \frac{1-aU}{U}\|v_t\|^2\nonumber\\
 & =& \frac{g}{U}\left({\rm Re}\int_\Omega \phi_t \overline
v_t dx- \frac{d_r}{d_i}{\rm Im}\int_\Omega \phi_t \overline v_t
dx\right)  -b\left({\rm Re}\int_\Omega (|v|^2v)_t\overline v_t
dx-\frac{d_r}{d_i} {\rm Im}\int_\Omega (|v|^2v)_t\overline v_t
dx\right)\nonumber\\
&:=&I_1+I_2,\nonumber
 \end{eqnarray}
 where
 \beq
 I_1\leq C\|\phi_t\|\|v_t\|\leq
 \frac{\gamma}{2}\|\phi_t\|^2+C\|v_t\|^2,
 \eeq
 \beq
 I_2\leq C\|v\|_{\mathcal{L}^6}^2\|v_t\|_{\mathcal{L}^3}^2\leq
 C\|v\|_{\mathcal{H}^1}^2\|\nabla v_t\|\|v_t\|\leq \frac{c}{8m}\|\nabla
 v_t\|^2+C\|v\|_{\mathcal{H}^1}^4\|v_t\|^2.
 \eeq
As a result,
 \beq
 \left(\frac{d_i}{2}+\frac{d_r^2}{2d_i}\right)\frac{d}{dt}\|v_t\|^2+\frac{c}{8m}\|\nabla v_t\|^2+
 \frac{1-aU}{U}\|v_t\|^2\leq
\frac{\gamma}{2}\|\phi_t\|^2+C\|v\|_{\mathcal{H}^1}^4\|v_t\|^2.\label{vvtb}
 \eeq
On the other hand, differentiating the $\phi$-equation in \eqref{3}
with respect to $t$, multiplying the resultant by
$\overline{\phi}_t$, integrating over $\Omega$ and taking the real
part, in analogy to \eqref{dphic} we obtain that
 \beq
 \frac{d}{dt}\|\phi_t\|^2+{\gamma}\|\phi_t\|^2\leq \frac{g^2}{\gamma U^2}\|v_t\|^2.\label{ppt}
 \eeq
 By \eqref{vvtb}, \eqref{ppt} and Cauchy-Schwarz inequality, we get
 \beq
 \frac{d}{dt}\left[\left(d_i+\frac{d_r^2}{d_i}\right)\|v_t\|^2+2\|\phi_t\|^2\right]+\frac{c}{4m}\|\nabla
 v_t\|^2+{\gamma}\|\phi_t\|^2\leq
 \left(\frac{2g^2}{\gamma U^2}+C\|v\|_{\mathcal{H}^1}^4\right)\|v_t\|^2.\label{hie}
 \eeq
It follows from \eqref{uniH1} that there exists $t_0=t_0(r_0)>0$
such that $\|v(t)\|_{\mathcal{H}^1}\leq M$ for $t\in[0, t_0]$ and
$\|v(t)\|_{\mathcal{H}^1}\leq M'$ for all $t\geq t_0$, with $M$ being a
constant depending on $r_0$ while $M'$ being independent of $r_0$.
Thus, on $[0,t_0]$, \eqref{hie} implies that
 \beq
 \frac{d}{dt}\left[\left(d_i+\frac{d_r^2}{d_i}\right)\|v_t\|^2+2\|\phi_t\|^2\right]\leq
 \left(\frac{2g^2}{\gamma U^2}+CM^4\right)\|v_t\|^2,
 \eeq
 which together with the Gronwall inequality yields
 \beq
 \|v_t(t_0)\|^2+\|\phi_t(t_0)\|^2\leq M_1,\label{vpta}
 \eeq
 where $M_1$ depends on $r_1$, $r_0$ and $t_0$.
\noindent  Let us start from time $t_0$. We infer from \eqref{hie}
that
  \beq
 \frac{d}{dt}\left[\left(d_i+\frac{d_r^2}{d_i}\right)\|v_t\|^2+2\|\phi_t\|^2\right]+\frac{c}{4m}\|\nabla
 v_t\|^2+{\gamma}\|\phi_t\|^2\leq K_1 \|v_t\|^2, \quad \forall \ t\geq t_0,
 \label{hiea}
 \eeq
where $K_1=\frac{2g^2}{\gamma U^2}+C(M')^4$. Denote
 \beq
 E_3(t)=\left(d_i+\frac{d_r^2}{d_i}\right)\|v_t(t)\|^2+2\|\phi_t(t)\|^2+\frac{4K_1+4}{d_i} E_1(t).\label{E302}
 \eeq
 Then it follows from \eqref{ABO1} and  \eqref{hiea} that
 \beq
 \frac{d}{dt}E_3(t)+K_2E_3(t)\leq K_3, \quad \forall t\geq t_0,
 \eeq
which yields
 \beq
 E_3(t)\leq e^{-K_2t}e^{K_2t_0}E_3(t_0)+\frac{K_3}{K_2}, \quad \forall t\geq t_0.\label{uniE3}
 \eeq
 From \eqref{E302} and \eqref{vpta}, we know that $E_3(t_0)$
 can be bounded by a constant depending on $r_1$, $r_0$ and $t_0$. Then
 it follows from \eqref{uniE3} that  there exists
 a time $t_1 \geq t_0$  depending on $r_1$, $r_0$ and $t_0$ such that
 \beq
 \left(d_i+\frac{d_r^2}{d_i}\right)\|v_t(t)\|^2+2\|\phi_t(t)\|^2\leq E_3(t)\leq \frac{2K_3}{K_2},\label{est}
 \quad \forall t\geq t_1.
 \eeq
On the other hand, we deduce from \eqref{3} that
 \begin{eqnarray}
 && \|v(t)\|_{\cH^2}\leq C
  (\|v_t(t)\|+\|\phi(t)\|+\|v(t)\|+\|v(t)\|_{\mathcal{L}^6}^3+\|f\|),\nonumber\\
  && \|\phi(t)\|^2_{\cH^2}\leq
  C(\|\phi_t(t)\|+\|\phi(t)\|+\|v(t)\|+\|h\|),\nonumber
 \end{eqnarray}
 where $C$ is a constant depending only on the coefficients of
 system \eqref{3}. Thus, from \eqref{est} and Proposition \ref{P3.1},
 we can see that there exists a constant $R_1>0$ independent of $r_0, r_1$ such that
 $$\sup\limits_{(v_0, \phi_0)\in \mathcal{B}}\sup_{t\geq t_1} \|S(t)(v_0,\phi_0)\|_{\cH^2\times\cH^2}\leq R_1.$$
The proof is complete.
\end{proof}
As a byproduct, the above lemma gives the following integral
estimate
\begin{corollary}\label{cor1}
There holds
\beq\label{integral}\sup\limits_{\|(v,\phi)\|_{\cH^2\times\cH^2}\leq
R}\sup\limits_{t\geq0}\int_t^{t+1}(\|v_t\|_{\cH^1}^2+\|\phi_t\|^2)d\tau\leq
C(R).\eeq
\end{corollary}
Next, we prove the following proposition that enables us to confine
the dynamics of system \eqref{3}-\eqref{4} to a regular set
$\tilde{\mathcal{B}}_1\subset\ (\cH^2\cap
\cH_0^1)\times(\cH^2\cap\cH^1_0)$.

\begin{proposition}\label{P5.2}
There exists a closed ball $\tilde{\mathcal{B}}_1\subset\ (\cH^2\cap \cH_0^1)\times(\cH^2\cap\cH^1_0)$
such that\\
(i) there is a positive increasing function $M$ such that for every bounded set $\mathcal{B}\subset \cH_0^1\times \cH_0^1$
with $R=\sup\limits_{(v,\phi)\in\mathcal{B}}\|(v,\phi)\|_{\cH_0^1\times \cH_0^1}$, the following estimate holds:
 \beq \mathrm{dist}_{\cH_0^1\times \cH_0^1}(S(t)\mathcal{B},\tilde{\mathcal{B}}_1)\leq M(R)e^{-\gamma t};
 \eeq
(ii) there is a time $\tilde{t}_1\geq0$ depending on
$\tilde{R}=\sup\limits_{(v,\phi)\in\tilde{\mathcal{B}}_1}\|(v,\phi)\|_{\cH^2\times
\cH^2}$ such that
 \beq S(t)\tilde{\mathcal{B}}_1\subset\tilde{\mathcal{B}}_1,\qquad\forall t\geq \tilde{t}_1.\eeq
\end{proposition}
\begin{proof}
 The existence of a bounded exponential attracting ball
 $\mathcal{G}\subset  (\cH^2\cap \cH_0^1)\times(\cH^2\cap\cH^1_0)$ in the $\cH_0^1\times \cH_0^1$ metric is
 given by Lemmas \ref{P4.1}-\ref{P4.3}.
 Next, using Proposition \ref{P5.3}, we can enlarge the ball $\mathcal{G}$ properly such that under the
 action of $S(t)$, after a time $t_1=t_1(\mathcal{G})$, $\mathcal{G}$ is absorbed into itself.
 Taking $\tilde{\mathcal{B}}_1=\mathcal{G}$, we complete the proof.
\end{proof}

We can now state the main result of this section:

\begin{theorem}\label{EPA}
The semigroup $S(t)$ possesses an exponential attractor
$\mathcal{E}\subset (\cH^2\cap \cH_0^1)\times(\cH^2\cap\cH^1_0)$.
Thus, by definition , we have that\\
(i) $\mathcal{E}$ is a closed compact set in $\cH_0^1\times\cH^1_0$
that is positively invariant for $S(t$).\\
(ii) The fractal dimension of $\mathcal{E}$ is finite.\\
(iii) $\mathcal{E}$ satisfies the following exponential attraction
property: there exist a constant $\omega > 0$ and a positive
increasing function $J$ such that, for every bounded set
$\mathcal{B}\subset (\cH^2\cap \cH_0^1)\times(\cH^2\cap\cH^1_0)$
with $R =
\sup_{(v,\phi)\in\mathcal{B}}\|(v,\phi)\|_{\cH^1_0\times\cH^1_0}$,
it holds \beq {\rm dist}_{\cH^1_0\times\cH^1_0}(S(t)\mathcal{B},
\mathcal{E})\leq J(R)e^{-\omega t}, \quad \forall t \geq
0.\label{exp} \eeq
\end{theorem}
\begin{proof}
The proof of Theorem \ref{EPA} consists of several steps.

\textbf{Step 1.}  First, we confine the dynamics of system
\eqref{3}-\eqref{4} to the regular set
$\tilde{\mathcal{B}}_1\subset\ (\cH^2\cap
\cH_0^1)\times(\cH^2\cap\cH^1_0)$ obtained in Proposition
\ref{P5.2}. In order to prove the existence of an exponential
attractor, we shall use the simple constructive method introduced in
\cite[Proposition 1]{EMZ} and follow the strategy in \cite{BG} (cf.
also \cite{GGP}). For the reader's convenience, we report the
following lemma adapted to our present case (cf. \cite[Lemma
5.3]{BG}).
\begin{lemma}
\label{lm5} Let $\tilde{\mathcal{B}}_1$ and $\tilde{t}_1$  be as in
Proposition \ref{P5.2} and denote $z=(v,\phi)$.
Suppose that there exists $t^*\geq \tilde{t}_1$ such that the following conditions are satisfied:\\
(C1) The map $
 (t,z)\mapsto S(t)z: [t^*,2t^*]\times\tilde{\mathcal{B}}_1\rightarrow\tilde{\mathcal{B}}_1
 $
 is $\frac{1}{2}$-H\"{o}lder continuous in time and Lipschitz continuous in the initial
 data, when $\tilde{\mathcal{B}}_1$ is endowed with the $\cH_0^1\times\cH^1_0$-topology.\\
(C2) Setting $S=S(t^*),$ there are $\lambda\in(0,\frac{1}{2})$ and
$\Lambda\geq0$ such that, for every $z_{01},$
$z_{02}\in\tilde{\mathcal{B}}_1$, $
Sz_{01}-Sz_{02}=D(z_{01},z_{02})+K(z_{01},z_{02}),$ where
$$\|D(z_{01},z_{02})\|_{\cH^1_0\times\cH^1_0}\leq\lambda\|z_{01}-z_{02}\|_{\cH^1_0\times\cH^1_0},\quad
\|K(z_{01},z_{02})\|_{\cH^2\times\cH^2}\leq
\Lambda\|z_{01}-z_{02}\|_{\cH^1_0\times\cH^1_0}.$$ Then there exists
a bounded set $\mathcal{E}\subset\tilde{\mathcal{B}}_1$, closed and
of finite fractal dimension in $\cH^1_0\times\cH^1_0$, positively
invariant for $S(t)$, such that for some $\omega_0>0$ and
$J_0\geq0$, it holds
 \beq
 \mathrm{dist}_{\cH^1_0\times\cH^1_0}(S(t)\tilde{\mathcal{B}}_1,\mathcal{E})\leq J_0 e^{-\omega_0 t}.\label{exp1}
 \eeq

\end{lemma}
It will be shown in the appendices that the conditions $(C1)$ and
$(C2)$ in Lemma \ref{lm5} are satisfied when the dynamics of system
\eqref{3}-\eqref{4} is confined to the regular set
$\tilde{\mathcal{B}}_1$. Hence, there exists a set
$\mathcal{E}\subset\tilde{\mathcal{B}}_1$, closed and of finite
fractal dimension in $\cH^1_0\times\cH^1_0$, positively invariant
for $S(t)$ and satisfying \eqref{exp1}.

\textbf{Step 2.} In order to complete the proof, we are left to show that \eqref{exp1} actually
holds for any bounded subset $\mathcal{B} \subset (\cH^2\cap
\cH_0^1)\times(\cH^2\cap\cH^1_0)$ instead of $\tilde{\mathcal{B}}_1$, with possibly
different $J_0$ and $\omega_0$. In other words, we have to prove that the basin of
exponential attraction can be the whole space $(\cH^2\cap
\cH_0^1)\times(\cH^2\cap\cH^1_0)$ (cf. \eqref{exp}).

For any bounded set $\mathcal{B}\subset (\cH^2\cap
\cH_0^1)\times(\cH^2\cap\cH^1_0)$ with $R = \sup_{(v,\phi)\in\mathcal{B}}\|(v,\phi)\|_{\cH^1_0\times\cH^1_0}$,
 it follows from Proposition \ref{P5.2} that
 \beq
\mathrm{dist}_{\cH_0^1\times \cH_0^1}(S(t)\mathcal{B},\tilde{\mathcal{B}}_1)\leq M(R)e^{-\gamma t}.\label{lll}
 \eeq
 On the other hand, for any $z_{01}=\left(v_0^{(1)},\phi_0^{(1)}\right)$, $z_{02}=\left(v_0^{(2)},\phi_0^{(2)}\right)
 \in\mathcal{B}$, by Theorem \ref{cod} (and Poincar\'e inequality), we have
 \beq
 \|S(t)z_{01}-S(t)z_{02}\|_{\cH^1_0\times\cH^1_0}
 \leq C_PL_1^\frac12e^{\frac{L_2}{2}t}\|z_{01}-z_{02}\|_{\cH^1_0\times\cH^1_0},\label{conss}
 \eeq
 where $C_P>0$ depends only on $\Omega$. Applying the abstract result on the transitivity of exponential
attraction (cf. \cite[Theorem 5.1]{FGMZ}), we conclude from
\eqref{exp1}, \eqref{lll} and \eqref{conss} that \beq
\mathrm{dist}_{\cH_0^1\times
\cH_0^1}(S(t)\mathcal{B},\mathcal{E})\leq Je^{-\omega t},
 \eeq
where
$$
J=J(R)=C_PL_1^\frac12M(R)+J_0,\quad \omega=\frac{\gamma
\omega_0}{\frac12 L_2+\gamma+\omega_0}.
$$
The proof is complete.
\end{proof}

We note that the exponential attractor $\mathcal{E}$ actually
contains the global attractor $\mathcal{A}$ that is obtained in
Section 3. As a consequence, we have
\begin{corollary}\label{GA1}
 The global attractor $\mathcal{A}$ has finite fractal dimension.
\end{corollary}

\section{Appendices}
\setcounter{equation}{0} We verify the conditions $(C1)$ and $(C2)$
in Lemma \ref{lm5} when the dynamics of system \eqref{3}-\eqref{4}
is confined to the regular set $\tilde{\mathcal{B}}_1$.

(1) \textbf{Verifying condition} $(C1)$.

\noindent For any $t,\tau\in [t^*, 2t^*]$ satisfying $t\geq \tau$,
we take the difference of the $\phi$-equation:
 \begin{eqnarray}
 &&\phi_t(t)-\phi_t(\tau)+\gamma(\phi(t)-\phi(\tau))-\frac{ig}{U}(v(t)-v(\tau))
+i\left(\frac{g^2}{U}+2\nu-2\mu\right)(\phi(t)-\phi(\tau))\nonumber\\
&& -\frac{i}{4m}(\Delta\phi(t)-\Delta\phi(\tau))=0.\nonumber
 \end{eqnarray}
Multiplying it by $\overline{\phi}(t)-\overline{\phi}(\tau)$, integrating over $\Omega$ and taking the
imaginary part, we obtain
 \begin{eqnarray}
&&\frac{1}{4m}\|\nabla\phi(t)-\nabla\phi(\tau)\|^2\nonumber\\
&=&-\left(\frac{g^2}{U}+2\nu-2\mu\right)\|\phi(t)-\phi(\tau)\|^2+\frac{g}{U}\mathrm{Re}
\int_{\Omega}(v(t)-v(\tau))(\overline{\phi}(t)-\overline{\phi}(\tau))dx
\nonumber\\
&& -\mathrm{Im}\int_{\Omega}(\phi_t(t)-\phi_t(\tau))(\overline{\phi}(t)-\overline{\phi}(\tau))dx\nonumber\\
&\leq&C(\|\phi(t)-\phi(\tau)\|^2+\|v(t)-v(\tau)\|^2
+\|\phi_t(t)-\phi_t(\tau)\|\|\phi(t)-\phi(\tau)\|). \label{est1p}
 \end{eqnarray}
  By \eqref{uniH1} and \eqref{est}, we know that for $t\geq \tilde{t}_1$,
  $\|\phi(t)\|_{\cH^1}$, $\|v(t)\|_{\cH^1}$, $\|\phi_t(t)\|$ and $\|v_t(t)\|$
  can be uniformly bounded by a constant independent of the initial data.
  Then we infer from \eqref{est1p} that
\begin{eqnarray}
&&\|\nabla \phi(t)-\nabla \phi(\tau)\|^2\nonumber\\
&\leq& C(\|\phi(t)\|+ \|\phi(\tau)\|+\|v(t)\|+\|v(\tau)\|+\|\phi_t(t)\|+
 \|\phi_t(\tau)\|)(\|v(t)-v(\tau)\|+\|\phi(t)-\phi(\tau)\|)\nonumber\\
&\leq& C\left(\int_\tau^t\|v_t(s)\|ds+\int_\tau^t\|\phi_t(s)\|ds\right)\nonumber\\
&\leq& C(t-\tau).
\end{eqnarray}
This and the Poincar\'e inequality yield that
 \beq
 \|\phi(t)- \phi(\tau)\|_{\cH^1_0}^2\leq C(t-\tau).\label{lip}
 \eeq
On the other hand, it follows from Corollary \ref{cor1} that
\beq\label{vvv}\|v(t)-v(\tau)\|_{\cH^1_0}\leq\int_{\tau}^t\|v_t(s)\|_{\cH^1_0}ds
\leq\bigg(\int_{\tau}^t\|v_t(t)\|_{\cH^1_0}^2ds\bigg)^{\frac{1}{2}}\sqrt{t-\tau}.
 \eeq
Denote $z=(v,\phi)$. For any $t^*\geq \tilde{t}_1$, $t, \tau \in
[t^*, 2t^*]$ with $t\geq \tau$, and
$z_1,z_2\in\tilde{\mathcal{B}}_1$, we infer from \eqref{contii},
\eqref{lip} and \eqref{vvv} that
\begin{eqnarray}\|S(t)z_1-S(\tau)z_2\|_{\cH^1_0\times\cH^1_0}&\leq&\|S(t)z_1-S(t)z_2\|_{\cH^1_0\times\cH^1_0}
+\|S(t)z_2-S(\tau)z_2\|_{\cH^1_0\times\cH^1_0}\nonumber\\
&\leq&
C(t^*)\bigg(\|z_1-z_2\|_{\cH^1_0\times\cH^1_0}+\sqrt{t-\tau}\bigg).\nonumber
\end{eqnarray}

(2) \textbf{Verifying condition} $(C2)$.

\noindent For any initial data
$z_{01}=\left(v_0^{(1)},\phi_0^{(1)}\right)$,
$z_{02}=\left(v_0^{(2)},\phi_0^{(2)}\right)\in\tilde{\mathcal{B}}_1$,
we set
$z_0=(v_0,\phi_0):=\left(v_0^{(1)}-v_0^{(2)},\phi_0^{(1)}-\phi_0^{(2)}\right)$.
The difference of the solutions $S(t)z_{0j}=(v^{(j)}, \phi^{(j)})$,
$j=1,2$ can be decomposed as
$$(v,\phi):=(v^{(1)}-v^{(2)}, \phi^{(1)}-\phi^{(2)})=(v^d,\phi^d)+(v^c,\phi^c),$$
where $(v^d,\phi^d)$ solves the linear problem
 \beq
 \begin{cases}\label{ddd}
 dv^d_t-i\left(a-\frac{1}{U}\right)v^d-\frac{ig}{U}\phi^d-\frac{ic}{4m}\Delta v^d=0,\\
 \phi^d_t+\gamma\phi^d+i\left(\frac{g^2}{U}+2\nu-2\mu\right)\phi^d-\frac{i}{4m}\Delta\phi^d=0,\\
v^d|_\Gamma=\phi^d|_{\Gamma}=0,\\
v^d(0)=v_0,\quad \phi^d(0)=\phi_0,\end{cases}
 \eeq
while $(v^c,\phi^c)$ satisfies
 \beq\label{jjj}
 \begin{cases}dv^c_t-i\left(a-\frac{1}{U}\right)v^c-\frac{ig}{U}\phi^c-\frac{ic}{4m}
 \Delta v^c+ib|v^{(1)}|^2v^{(1)}-ib|v^{(2)}|^2v^{(2)}=0,\\
\phi^c_t+\gamma\phi^c-\frac{ig}{U}v+i\left(\frac{g^2}{U}+2\nu-2\mu\right)\phi^c-\frac{i}{4m}\Delta\phi^c=0,\\
v^c|_{\Gamma}=\phi^c|_{\Gamma}=0,\\
v^c(0)=0,\quad\phi^c(0)=0.\end{cases}
 \eeq
 Similar to \eqref{cccc}, we have
 \beq
 \frac{1}{2}\frac{d}{dt}\|\phi^d\|_{\cH^1}^2+\gamma
 \|\phi^d\|^2_{\cH^1}=0,
 \eeq
 which implies
 \beq
 \|\phi^d(t)\|^2_{\cH^1}\leq e^{-2\gamma t}\|\phi_0\|^2_{\cH^1}, \quad t\geq
 0.\label{decaypd}
 \eeq
Multiplying the first equation in \eqref{ddd} by $\overline
{v^d}+\alpha \overline {v^d_t}$ ($\alpha>0$), integrating over
$\Omega$ and taking the imaginary part, we have
 \begin{eqnarray}
&&
\frac{d}{dt}\left[\left(\frac{\alpha}{2}\left(\frac{1}{U}-a\right)+\frac{d_i}{2}\right)\|v^d\|^2+\frac{\alpha
c}{8m}\|\nabla v^d\|^2\right]+\alpha
d_i\|v^d_t\|^2\nonumber\\
&& +\left(\frac{1}{U}-a\right)\|v^d\|^2+\frac{c}{4m}\|\nabla
v^d\|^2 \nonumber\\
&=&\frac{\alpha
g}{U}\mathrm{Re}\int_{\Omega}\phi^d\overline{v^d_t}dx+\frac{g}{U}\mathrm{Re}\int_{\Omega}\phi^d\overline{v^d}dx-d_r
\mathrm{Im}\int_{\Omega}v^d_t\overline{v^d}dx\nonumber\\
&\leq& \frac{\alpha
d_i}{4}\|v^d_t\|^2+C_1\alpha\|\phi^d\|^2+\frac{1}{2}\left(\frac{1}{U}-a\right)\|v^d\|^2+C_2(\|\phi^d\|^2+\|v^d_t\|^2).
 \end{eqnarray}
Taking $\alpha=\frac{4C_2}{d_i}$ in the above inequality, we arrive
at
\begin{eqnarray}
&&
\frac{d}{dt}\left[\left(\frac{\alpha}{2}\left(\frac{1}{U}-a\right)+\frac{d_i}{2}\right)\|v^d\|^2+\frac{\alpha
c}{8m}\|\nabla v^d\|^2\right]+\frac{\alpha
d_i}{2}\|v^d_t\|^2\nonumber\\
&& +\frac12\left(\frac{1}{U}-a\right)\|v^d\|^2+\frac{c}{4m}\|\nabla
v^d\|^2 \leq (C_1\alpha+C_2)\|\phi^d\|^2.\label{decayvda}
 \end{eqnarray}
It easily follows from \eqref{decaypd} and \eqref{decayvda} that
there exists $\gamma_1\in (0, \gamma)$
 \begin{eqnarray}
  \|v^d(t)\|_{\cH^1}^2&\leq &e^{-2\gamma_1
 t}\|v_0\|_{\cH^1}^2+Ce^{-2\gamma_1t}\int_0^t
 e^{2\gamma_1\tau}\|\phi^d(\tau)\|^2d\tau\nonumber\\
 &\leq& e^{-2\gamma_1
 t}\|v_0\|_{\cH^1}^2+\frac{C}{(\gamma-\gamma_1)}e^{-2\gamma_1
 t}\|\phi_0\|^2_{\cH^1},\quad \forall t\geq 0,\label{decayvd}
 \end{eqnarray}
 and
 \beq
 \int_0^t\|v^d_t(\tau)\|^2d\tau \leq C\|v_0\|_{\cH^1}^2+C\int_0^t
 \|\phi^d(\tau)\|^2d\tau\leq
 C(\|v_0\|_{\cH^1}^2+\|\phi_0\|^2_{\cH^1}).\label{intvd}
 \eeq
 Then \eqref{decaypd}, \eqref{decayvd} and \eqref{contii} yield that
  \begin{eqnarray}
  \|v^c(t)\|_{\cH^1}+\|\phi^c(t)\|_{\cH^1}
   &\leq&
   \|v(t)\|_{\cH^1}+\|v^d(t)\|_{\cH^1}+\|\phi(t)\|_{\cH^1}+\|\phi^d(t)\|_{\cH^1}\nonumber\\
 &\leq& C(t) \|(v_0,\phi_0)\|_{\cH^1\times\cH^1}, \quad \forall t\geq
 0.\label{dilow}
 \end{eqnarray}
Next, we try to get higher-order estimate of $(v^c,\phi^c)$. For
this purpose, we take the time derivative of equations in
\eqref{jjj}:
 \begin{eqnarray}
 &&dv^c_{tt}-i\left(a-\frac{1}{U}\right)v^c_t-\frac{ig}{U}\phi^c_t-\frac{ic}{4m}\Delta v^c_t+ib(|v^{(1)}|^2
 v^{(1)}-ib|v^{(2)}|^2v^{(2)})_t=0,\label{dvt1}\\
 && \phi^c_{tt}+\gamma\phi^c_t-\frac{ig}{U}v_t+i\left(\frac{g^2}{U}+2\nu-2\mu\right)
 \phi^c_t-\frac{i}{4m}\Delta\phi_t^c=0.\label{dpt1}
 \end{eqnarray}
Multiplying \eqref{dvt1} by $\overline {v^c_t}$, integrating over
$\Omega$ and taking the imaginary part/real part, respectively, we
have
\begin{eqnarray}
 &&\frac{d_i}{2}\frac{d}{dt}\|v^c_t\|^2+\frac{c}{4m}\|\nabla v^c_t\|^2+
 \frac{1-aU}{U}\|v^c_t\|^2\nonumber\\
 & =& \frac{g}{U}{\rm Re}\int_\Omega \phi^c_t \overline
{v^c_t} dx-b{\rm Re}\int_\Omega
(|v^{(1)}|^2v^{(1)}-|v^{(2)}|^2v^{(2)})_t\overline {v^c_t}
dx-d_r{\rm Im}\int_\Omega v^c_{tt} \overline {v^c_t} dx,\label{C1}
 \end{eqnarray}
  \beq
 \frac{d_r}{2}\frac{d}{dt}\|v^c_t\|^2-d_i{\rm Im}\int_\Omega v^c_{tt} \overline
 {v^c_t}dx= -\frac{g}{U}{\rm Im}\int_\Omega \phi^c_t \overline
{v^c_t} dx+b{\rm Im}\int_\Omega
(|v^{(1)}|^2v^{(1)}-|v^{(2)}|^2v^{(2)})_t\overline {v^c_t}
dx.\label{C1a}
 \eeq
 As in the previous section, we can insert \eqref{C1a} into \eqref{C1}
 to cancel the higher-order term ${\rm Im}\int_\Omega v^c_{tt} \overline {v^c_t}
 dx$:
\begin{eqnarray}
 &&\left( \frac{d_i}{2}+\frac{d_r^2}{2d_i}\right)\frac{d}{dt}\|v^c_t\|^2+\frac{c}{4m}\|\nabla v^c_t\|^2+
 \frac{1-aU}{U}\|v^c_t\|^2\nonumber\\
 & =& \frac{g}{U}\left({\rm Re}\int_\Omega \phi^c_t \overline
{v^c_t} dx -  \frac{d_r}{d_i}{\rm Im}\int_\Omega \phi^c_t \overline
{v^c_t} dx\right)\nonumber\\
 &&-b\left({\rm Re}\int_\Omega
(|v^{(1)}|^2v^{(1)}-|v^{(2)}|^2v^{(2)})_t\overline {v^c_t}
dx-\frac{d_r}{d_i}{\rm Im}\int_\Omega
(|v^{(1)}|^2v^{(1)}-|v^{(2)}|^2v^{(2)})_t\overline {v^c_t}
dx\right)\nonumber\\
&:=&J_1+J_2.\label{C11}
 \end{eqnarray}
 \beq
 J_1\leq C\|\phi^c_t\|\|v^c_t\|\leq \frac{\gamma}{2}\|\phi^c_t\|^2+C\|v^c_t\|^2.\label{C2}
 \eeq
 By Proposition \ref{P5.3}, we know that
 \beq
 \|v^{(j)}(t)\|_{\cH^2}+\|\phi^{(j)}(t)\|_{\cH^2}\leq C, \quad j=1,2, \ \forall t\geq 0.\label{uniHH2}
 \eeq
 It easily  follows from \eqref{uniHH2}, the Sobolev embedding theorem and the Poincar\'e
 inequality that
 \begin{eqnarray}
 J_2&\leq& C(\|v^{(1)}\|_{\mathcal{L}^\infty}+\|v^{(2)}\|_{\mathcal{L}^\infty})
  \|v\|_{\mathcal{L}^6}\|v^{(1)}_t\|\|v^c_t\|_{\mathcal{L}^3}+
 C\|v^{(2)}\|^2_{\mathcal{L}^\infty}\|v_t\|\|v^c_t\|\nonumber\\
  &\leq& C\|v\|_{\cH^1}\|\nabla v^c_t\|^\frac12\|v_t^c\|^\frac12+C\|v_t\|\|v^c_t\|\nonumber\\
  &\leq& \frac{c}{8m}\|\nabla v^c_t\|^2+ C\|v^c_t\|^2+C\|v_t\|^2+C\|v\|_{\cH^1}^2.\label{C3}
 \end{eqnarray}
 We can conclude from\eqref{C11}, \eqref{C2} and \eqref{C3}  that
 \beq
 \left(\frac{d_i}{2}+\frac{d_r^2}{2d_i}\right)\frac{d}{dt}\|v^c_t\|^2+\frac{c}{8m}\|\nabla v^c_t\|^2\leq
 \frac{\gamma}{2}\|\phi^c_t\|^2+C\|v^c_t\|^2+C\|v_t\|^2+C\|v\|_{\cH^1}^2.\label{CC}
 \eeq
On the other hand, similar to \eqref{ppt}, there holds
 \beq
 \frac{d}{dt}\|\phi^c_t\|^2+\gamma\|\phi^c_t\|^2\leq \frac{g^2}{\gamma U^2}\|v_t\|^2.\label{C4}
 \eeq
 Then it follows from \eqref{CC} and \eqref{C4} that
 \beq
 \frac{d}{dt}\left[\left(\frac{d_i}{2}+\frac{d_r^2}{2d_i}\right)\|v^c_t\|^2+\|\phi^c_t\|^2\right]
 \leq C\|v^c_t\|^2+C\|v_t\|^2+C\|v\|_{\cH^1}^2.
 \eeq
Integrating with respect to time, using \eqref{contii} and
\eqref{dilow}, we get
 \begin{eqnarray}
 && \left(\frac{d_i}{2}+\frac{d_r^2}{2d_i}\right)\|v^c_t(t)\|^2+\|\phi^c_t(t)\|^2
 \leq
 C\int_0^t(\|v(\tau)\|^2_{\cH^1}+\|v_t(\tau)\|^2+\|v^c_t(\tau)\|^2)d\tau\nonumber\\
 &\leq&
 C\int_0^t(\|v(\tau)\|^2_{\cH^1}+\|v_t(\tau)\|^2+\|v^d_t(\tau)\|^2)d\tau\leq
  C(t)(\|v_0\|_{\cH^1_0}^2+\|\phi_0\|^2_{\cH^1_0}).\label{dihigh}
 \end{eqnarray}
 By \eqref{contii}, \eqref{dilow}, \eqref{dihigh} and the Sobolev embedding theorem, we deduce from equation
 \eqref{jjj} and the elliptic regularity theorem that
 \begin{eqnarray}
 \|v^c(t)\|_{\cH^2}+\|\phi^c(t)\|_{\cH^2}&\leq&
 C(\|v^c_t(t)\|+\|\phi^c_t(t)\|+\|v(t)\|_{\cH^1}+\|v^c(t)\|+\|\phi^c(t)\|)\nonumber\\
 &\leq& C(t)(\|v_0\|_{\cH^1_0}+\|\phi_0\|_{\cH^1_0}).\label{KK}
 \end{eqnarray}
 Due to \eqref{decaypd} and \eqref{decayvd}, for any fixed $\lambda\in (0,\frac12)$,
 we can choose $t^*\geq \tilde{t}_1$ sufficiently large such
 that
 \beq
  \|v^d(t^*)\|^2_{\cH^1_0}+\|\phi^d(t^*)\|^2_{\cH^1_0}
  \leq \lambda^2(\|v_0\|_{\cH^1_0}^2+\|\phi_0\|^2_{\cH^1_0}).\label{DD}
 \eeq
 Set
 \beq
 D(z_{01},z_{02})=(v^d(t^*),\phi^d(t^*)),\quad
 K(z_{01},z_{02})=(v^c(t^*),\phi^c(t^*)).
 \eeq
 It follows from \eqref{KK} and \eqref{DD} that condition $(C2)$ are
 satisfied.

%%%%%%%%%%%%%%%%%%%%%%%%%%%%%%%%%%%%%%%%%%%%%%%%%%%%
\bigskip

\textbf{Acknowledgement.} The authors want to thank Prof. S. Zheng
for his helpful discussions. The research of H. Wu was partially
supported by Natural Science Foundation of China 11001058.

 \end{document}